\documentclass[article,11pt]{amsart}
\usepackage[top=25mm,bottom=25mm,left=25mm,right=25mm]{geometry}
\usepackage{graphicx}
\usepackage{float}
\usepackage{amssymb}

\usepackage{color}

\newtheorem{lemma}{Lemma}[section]

\newtheorem{theorem}{Theorem}[section]
\newtheorem{corollary}{Corollary}[section]
\newtheorem{definition}{Definition}[section]

\theoremstyle{definition}

\makeatletter
\def\section{\@startsection{section}{1}%
\z@{1\linespacing\@plus\linespacing}{1\linespacing}%
{\bf\centering}}
\def\subsection{\@startsection{subsection}{0}%
\z@{\linespacing\@plus\linespacing}{\linespacing}%
{\bf}}
\makeatother

\makeatletter
\@addtoreset{equation}{section}
\makeatother

\DeclareMathOperator{\diam}{diam}

\newcommand{\cB}{\mathcal{B}}

\begin{document}
\title[Estimates of densities for the reflected Brownian motion on nested fractals]
{Estimates of the transition densities for the reflected Brownian motion on simple nested fractals}
\author{Mariusz Olszewski}

\address{M. Olszewski \\ Faculty of Pure and Applied Mathematics, Wroc{\l}aw University of Science and Technology, Wyb. Wyspia\'nskiego 27, 50-370 Wroc{\l}aw, Poland}
\email{mariusz.olszewski@pwr.edu.pl}

\begin{abstract}
{We give sharp two-sided estimates for the functions $g_M(t,x,y)$ and $g_M(t,x,y)-g(t,x,y)$, where $g_M(t,x,y)$ are the transition probability densities of the reflected Brownian motion on a $M$-complex of size $M \in \mathbb{Z}$ of an unbounded planar simple nested fractal and $g(t,x,y)$ are the transition probability densities of the `free' Brownian motion on this fractal. This is done for a large class of planar simple nested fractals with the good labeling property.
}


\bigskip
\noindent
2010 {\it MS Classification}: {Primary: 60J35, 28A80; Secondary: 60J25, 60J65} 
\end{abstract}

\footnotetext{M. Olszewski was supported in part by the National Science Center, Poland, grant no. 2015/17/B/ST1/01233}

\maketitle

\baselineskip 0.5 cm

\bigskip\bigskip

\section{Introduction}

The analysis and probability theory (especially stochastic processes) on fractals underwent rapid development over the last decades
(see e.g. \cite{bib:Bar, bib:Ki, bib:St1, bib:St2} and  the  references therein). The original motivation came from the investigations on the properties of disordered media in mathematical physics. Fractals also help us to understand the features of natural phenomena such as polymers, and growth of molds and crystals. The rigorous definition of the Brownian motion on the Sierpi\'nski gasket has been given by Barlow and Perkins in \cite{bib:BP} (see also \cite{bib:G, bib:Kus2}). Lindstr\o m \cite{bib:Lin} used a nonstandard analysis to construct such a process on general simple nested fractals (see also \cite{bib:F, bib:Kus, bib:Sh} for a Dirichlet form approach). The case of more general fractals was also addressed in \cite{bib:BBa, bib:Kum1, bib:KumS}. The estimates of the transition densities for the Brownian motion on simple nested fractals were proven by Kumagai in \cite{bib:Kum}. The case of more general class of finitely ramified fractals (called affine nested fractals) was studied in \cite{bib:FHK}. 

The present article is a companion paper to \cite{bib:KOPP}, where the \emph{reflected Brownian motion} in a $M$-complex, $M \in \mathbb{Z}$,  of a simple nested fractal was constructed (see also the previous paper \cite{bib:KPP-PTRF} for the case of the Sierpi\'nski triangle). Such a process was obtained as a 'folding' projection of the free Brownian motion from the unbounded fractal and the whole construction was performed under the key geometric assumption that the fractal has a \emph{good labelling property} (see Section 2 for more details). It was also proven in the cited paper that the one dimensional distributions of this process possess the continuous and symmetric densities $g_M(t,x,y)$ (see \eqref{eq:refldens} for a definition), which provided us with further regularity properties of the reflected process.

Our main goal in this paper is to find the sharp estimates of the densities $g_M(t,x,y)$ and for differences $g_M(t,x,y)-g(t,x,y)$, where $g(t,x,y)$ are the transition probability densities of a `free' process. We give an argument which allows us to deduce the two-sided sharp estimates for these functions from the intrinsic growth property of the graph metric on the planar nested fractal (for the definition of the graph metric $d_M(x,y)$ see \eqref{def:graphmetric}). More precisely, we show that there exist the positive constants $c_1,...,c_6$ (uniform in $t$, $x$, $y$ and $M$) such that 
\begin{align*}
c_1 \left( f_{c_2}(t,|x-y|) \vee h_{c_3}(t,M)\right) &\leq g_M(t,x,y)\\
& \leq c_4 \left( f_{c_5}(t,|x-y|) \vee h_{c_6}(t,M)\right),
\end{align*} 
where 
\begin{align*}
f_c (t,r) = & \ t^{-d_f/d_w} \exp \left(-c \left(\frac{r^{d_w}}{t} \right)^{\frac{1}{d_J -1}} \right)
\end{align*}
and
\begin{align*}
h_c (t,M) = & \ L^{-d_{f}M}  \left( \frac{L^M}{t^{1/d_w}} \vee 1\right)^{d_f - \frac{d_w}{d_J -1}} \exp\left(-c \left( \frac{L^M}{t^{1/d_w}} \vee 1\right)^{\frac{d_w}{d_J -1}}\right).
\end{align*}
This result is given in Theorem \ref{th:main}. Here $L$ is a scaling factor of the fractal and $d_f$, $d_w$ and $d_J$ are certain parameters determined by the geometry of the fractal. One can see from the above estimates that for $t \geq L^{M d_w}$ the density $g_M(t,x,y)$ behaves like $L^{-Md_f}$, which shows that in large times the reflected process is distributed almost uniformly on a given $M$-complex. When $t < L^{M d_w}$, then the reflected process less 'feels' the reflection and resembles the free diffusion (for details see Corollary \ref{cor:uniform}). This effect is explained by our second main result (Theorem \ref{th:main2}), which gives the sharp two-sided estimates for the difference $g_M(t,x,y)-g(t,x,y)$. Indeed, we prove that there are positive constants $c_7,...,c_{12}$ (again uniform in $t$, $x$, $y$ and $M$) such that 
\begin{align*}
c_{7}  \big(f_{c_{8}}(t,\delta_M(x,y)) \vee h_{c_{9}}(t,M)\big)
& \leq g_M(t,x,y)-g(t,x,y) \\
& \leq c_{10} \big(f_{c_{11}}(t,\delta_M(x,y)) \vee h_{c_{12}}(t,M)\big),
\end{align*}
where $\delta_M(x,y) = \inf_{z \in V^{*}_{M}} (|x-z|+|z-y|)$ and $V^{*}_M$ is the set of all vertices of a basic $M$-complex that connect it with other $M$-complexes. In light of Corollary \ref{cor:uniform} mentioned above, this result can be understood as the second term asymptotic estimate of the density $g_M(t,x,y)$ for $t < L^{M d_w}$. It emerges that the dependence on the boundary of a given $M$-complex occurs only in the second term of this expansion.

We would like to emphasize that we have direct applications for the estimates obtained in the present paper. In recent articles \cite{bib:KaPP2, bib:KaPP}, the reflected Brownian motion was used to prove the existence and further asymptotic properties of the integrated density of states for subordinate Brownian motions evolving in presence of the Poissonian random field on the Sierpi\'nski triangle. The estimates of the densities were an essential tool there. Our present results will allow us to continue this fascinating research in the case when the configuration space is modeled by a general simple nested fractal (in this context, it is crucial that our estimates describes the behaviour of $g_M(t,x,y)$ not only in $x, y$ and $t$, but also in $M$). This is the subject of an ongoing project. 

At the end of the Introduction, let us say a few words about our methods. First note that our upper bound for the tail of the series in Lemma \ref{lem:tail} extends a similar result in \cite[Lem. 2.5]{bib:KaPP2} obtained for the reflected Brownian motion on the Sierpi\'nski gasket. 
The proof of that bound in an essential way uses the facts that the $M$-complexes of any size $M \in \mathbb{Z}$ of the gasket agree with the Euclidean balls $B(0,2^M)$ intersected with the fractal and that the geodesic (or the shortest path) metric is uniformly comparable to the Euclidean one. Such a comparabilty condition is also a common assumption in the papers dealing with subordinate Brownian motions on fractals having the $d$-set structure \cite{bib:BSS, bib:KKw}. This argument does not have an extension to the general nested fractals, for which the geodesic metric is typically not well defined. To overcome this difficulty, we propose a new approach based on an application of the graph metric of order $M$ and works well for all nested fractals. Our main contribution is the observation that the intrinsic growth property of the graph metric stated in Lemma \ref{lem:metrics} leads to the sharp estimates of densities $g_M(t,x,y)$. We also want to mention that the concluding part of the proof of the upper bound in Lemma \ref{lem:tail} follows the general ideas from the proof of \cite[Lem. 2.5]{bib:KaPP2}, while the basic estimate in Lemma \ref{lem:adjacent} is a completely new observation.

\section{Preliminaries}

\subsection{Planar simple nested fractals}

Consider a collection of similitudes $\Psi_i : \mathbb{R}^2 \to \mathbb{R}^2$ with a common scaling factor $L>1$  and a common isometry part $U,$ i.e. $\Psi_i(x) = (1/L) U(x) + \nu_i,$  where  $\nu_i \in \mathbb{R}^2$, $i \in \{1, ..., N\}.$ We shall assume $\nu_1 = 0$.
There exists a unique nonempty compact set $\mathcal{K}^{\left\langle 0\right\rangle}$ (called {\em the  fractal generated by the system} $(\Psi_i)_{i=1}^N$) such that $\mathcal{K}^{\left\langle 0\right\rangle} = \bigcup_{i=1}^{N} \Psi_i\left(\mathcal{K}^{\left\langle 0\right\rangle}\right)$.  As $L>1$, each similitude has exactly one fixed point and there are exactly $N$ fixed points of the transformations $\Psi_1, ..., \Psi_N$.

\begin{definition}[\textbf{Essential fixed point}]
A fixed point $x \in \mathcal{K}^{\left\langle 0\right\rangle}$ is an essential fixed point if there exists another fixed point $y \in \mathcal{K}^{\left\langle 0\right\rangle}$ and two different similitudes $\Psi_i$, $\Psi_j$ such that $\Psi_i(x)=\Psi_j(y)$.
\end{definition}

The set of all essential fixed points of transformations $\Psi_1, ..., \Psi_N$ is denoted by $V_{0}^{\left\langle 0\right\rangle}$. Clearly, $k:=\# V^{\left\langle 0\right\rangle}_{0} \leq N$. For the Sierpi\'nski gasket $k = N$, but there are many examples with $k <N$ (see Fig. 1).

\begin{definition}[\textbf{Simple nested fractal}]
\label{def:snf}
 The fractal $\mathcal{K}^{\left\langle 0 \right\rangle}$ generated by the system $(\Psi_i)_{i=1}^N$ is called a \emph{simple nested fractal (SNF)} if the following five conditions are met: \\
(1) $\# V_{0}^{\left\langle 0\right\rangle} \geq 2.$ \\
(2) \emph{(Open Set Condition)} There exists an open set $U \subset \mathbb{R}^2$ such that for $i\neq j$ one has $\Psi_i (U) \cap \Psi_j (U)= \emptyset$ and $\bigcup_{i=1}^N \Psi_i (U) \subseteq U$. \\
(3) \emph{(Nesting)} $\Psi_i\left(\mathcal{K}^{\left\langle 0 \right\rangle}\right) \cap \Psi_j \left(\mathcal{K}^{\left\langle 0 \right\rangle}\right) = \Psi_i \left(V_{0}^{\left\langle 0\right\rangle}\right) \cap \Psi_j \left(V_{0}^{\left\langle 0\right\rangle}\right)$ for $i \neq j$. \\
(4) \emph{(Symmetry)} For $x,y \in V_{0}^{\left\langle 0\right\rangle},$ let $S_{x,y}$ denote the symmetry with respect to the line bisecting the segment $\left[x,y\right]$. Then
$$
\forall i \in \{1,...,M\} \ \forall x,y \in V_{0}^{\left\langle 0\right\rangle} \ \exists j \in \{1,...,M\} \ S_{x,y} \left( \Psi_i \left(V_{0}^{\left\langle 0\right\rangle} \right) \right) = \Psi_j \left(V_{0}^{\left\langle 0\right\rangle} \right).
$$
(5) \emph{(Connectivity)} On the set $V_{-1}^{\left\langle 0\right\rangle}:= \bigcup_i \Psi_i \left(V_{0}^{\left\langle 0\right\rangle}\right)$ we define graph structure $E_{-1}$ as follows: $(x,y) \in E_{-1}$ if and only if $x, y \in \Psi_i\left(\mathcal{K}^{\left\langle 0 \right\rangle}\right)$ for some $i$.\\
Then the graph $(V_{-1}^{\left\langle 0\right\rangle},E_{-1} )$ is required to be connected.
\end{definition}

\bigskip

If $\mathcal{K}^{\left\langle 0 \right\rangle}$ is a simple nested fractal, then we denote
\begin{align} \label{eq:Kn}
\mathcal{K}^{\left\langle M\right\rangle} = L^M \mathcal{K}^{\left\langle 0\right\rangle}, \quad M \in \mathbb{Z},
\end{align}
and
\begin{align} \label{eq:Kinfty}
\mathcal{K}^{\left\langle \infty \right\rangle} = \bigcup_{M=0}^{\infty} \mathcal{K}^{\left\langle M\right\rangle}.
\end{align}
The set $\mathcal{K}^{\left\langle \infty \right\rangle}$ is the \textbf{unbounded simple nested fractal (USNF)}.

\begin{definition} \label{def:obj} Let $M\in\mathbb Z.$ \\
(1) \emph{$M$-complex}: \label{def:Mcomplex}
every set $\Delta_M \subset \mathcal{K}^{\left\langle \infty \right\rangle}$ of the form
\begin{equation} \label{eq:Mcompl}
\Delta_M  = \mathcal{K}^{\left\langle M \right\rangle} + \sum_{j=M+1}^{J} L^{j} \nu_{i_j}
\end{equation}
for some $J \geq M+1$, $\nu_{i_j} \in \left\{\nu_1,...,\nu_N\right\}$, is called an \emph{$M$-complex}. The set of all $M$-complexes in $\mathcal{K}^{\left\langle \infty \right\rangle}$ is denoted by $\mathcal{T}_M$. \\
(2) \emph{Vertices of an $M$-complex}: the set $V\left(\Delta_M\right) = L^{M} V^{\left\langle 0 \right\rangle}_0 + \sum_{j=M+1}^{J} L^{j} \nu_{i_j}$. \\
(3) \emph{Vertices of $\mathcal{K}^{\left\langle M \right\rangle}$}:
$$
V^{\left\langle M\right\rangle}_{M} = V\left(\mathcal{K}^{\left\langle M \right\rangle}\right) = L^M V^{\left\langle 0\right\rangle}_{0}.
$$
(4) \emph{Vertices of all $M$-complexes} inside a $(M+m)$-complex for $m>0$:
$$
V_M^{\langle M+m\rangle}= \bigcup_{i=1}^{N} V_M^{\langle M+m-1\rangle} + L^M \nu_i.
$$ \\
(5) \emph{Vertices of all 0-complexes} inside the unbounded nested fractal:
$$
V^{\left\langle \infty \right\rangle}_{0} = \bigcup_{M=0}^{\infty} V^{\left\langle M\right\rangle}_{0}.
$$ 
(6) \emph{Vertices of $M$-complexes} from the unbounded fractal:
$$
V^{\left\langle \infty \right\rangle}_{M} = L^{M} V^{\left\langle \infty \right\rangle}_{0}.
$$ 
(7) \emph{The unique $M$-complex containing} $x \in \mathcal{K}^{\left\langle \infty \right\rangle}\backslash V^{\left\langle \infty \right\rangle}_{M}$ is denoted by $\Delta_M (x)$.
\end{definition}

\begin{figure}[ht]
\centering
	\includegraphics[scale=0.04]{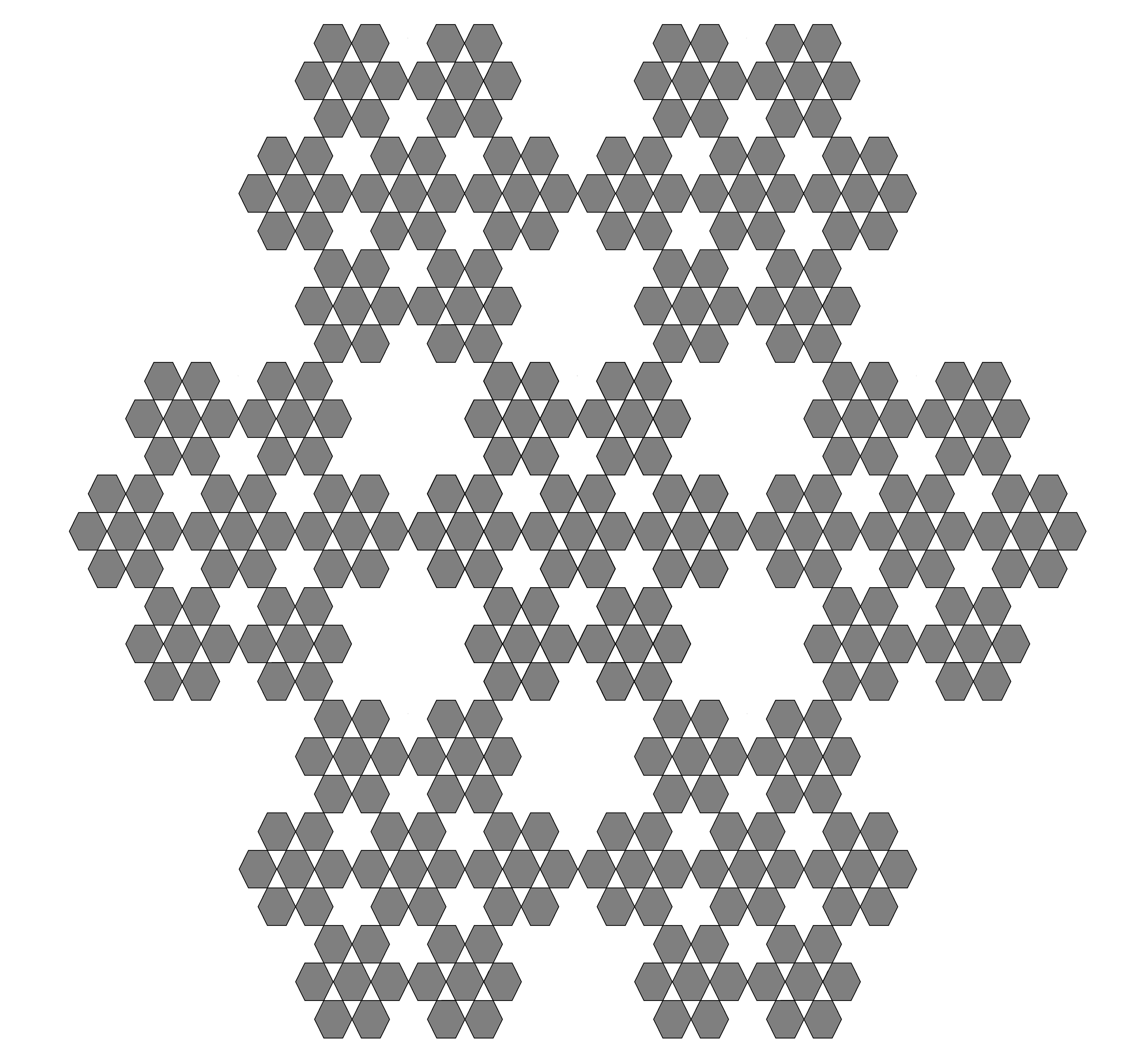}
\caption{An example of a nested fractal: the Lindstr\o m snowflake. It is constructed by 7 similitudes with $L=3$. It has 7 fixed points, but only 6 essential fixed points.}
\end{figure}

By $d_f$, $d_w$ and $d_s$ we denote the Hausdorff dimension, the walk dimension and the spectral dimension of SNF $\mathcal{K}^{\left\langle 0 \right\rangle}$, respectively. It is known that the identity $d_f/d_w = d_s/2$ holds.

The \emph{$M$-graph metric} on  $\mathcal{K}^{\left\langle \infty \right\rangle} \times \mathcal{K}^{\left\langle \infty \right\rangle}$  is defined as follows:
\begin{equation}
\label{def:graphmetric}
d_M (x,y):= \left\{ \begin{array}{ll}
0, & \textrm{if } x=y ;\\
1, & \textrm{if there exists } \Delta_M \in \mathcal{T}_M \textrm{ such that } x,y \in \Delta_M ;\\
n>1, & \textrm{if there does not exist } \Delta_M \in \mathcal{T}_M \textrm{ such that } x,y \in \Delta_M \\
&  \textrm{ and } n \textrm{ is the lowest number for which exist} \\
& \Delta_M^{(1)}, \Delta_M^{(2)}, ..., \Delta_M^{(n)} \in \mathcal{T}_M \textrm{ such that } x \in \Delta_M^{(1)},\\
&  y \in \Delta_M^{(n)} \textrm{ and } \Delta_M^{(i)} \cap \Delta_M^{(i+1)} \neq \emptyset \textrm{ for  }1 \leq i \leq n-1.
\end{array} \right.
\end{equation}

\bigskip

\subsection{Good labelling property and folding projections}

Throughout this section we assume that $M \in \mathbb{Z}$ is arbitrary but fixed. Note that
every $M$-complex $\Delta_M$ is a regular polygon with $k$ vertices \cite[Prop. 2.1]{bib:KOPP}. 
In consequence, there exist exactly $k$ different rotations $R_i$ around the barycenter of $\mathcal{K}^{\left\langle M \right\rangle}$,
mapping $\mathcal{K}^{\left\langle M \right\rangle}$ onto $\mathcal{K}^{\left\langle M \right\rangle}$ (for $i=1,2,...,k$ the rotation $R_i$ rotates by angle $(2\pi i)/k$). Denote $\mathcal{R}_M = \{R_1, ..., R_k\}$.

The concept of the \emph{good labelling property} (GLP in short) has been introduced in \cite{bib:KOPP}. Given the set of labels $\mathcal{A} = \{a_1, ..., a_k\}$, the labelling function is a map $\ell_M : V_{M}^{\left\langle \infty \right\rangle} \to \mathcal{A}$. It provides the \emph{good labelling} (of order $M$) if every $M$-complex has the complete set of labels mapped to its vertices and the vertices of any $M$-complex are labelled in the same orientation. More precisely:

\begin{itemize}
\item[(1)] For every $M$-complex $\Delta_M$ the restriction of $\ell_M$ to $V\left(\Delta_M\right)$ is a bijection onto $\mathcal{A}$.
\item[(2)] For every $M$-complex $\Delta_M$ of the form
$$
\Delta_M  = \mathcal{K}^{\left\langle M \right\rangle} + \sum_{j=M+1}^{J} L^{j} \nu_{i_j},
$$
with some $J \geq M+1$ and $\nu_{i_j} \in \left\{\nu_1,...,\nu_N\right\}$ (cf. Def. \ref{def:obj} (1)), there exists a rotation $R_{\Delta_M} \in \mathcal{R}_M$ such that
\begin{align} \label{eq:rotation}
\ell_M(v)=\ell_M\left(R_{\Delta_M}\left(v -\sum_{j=M+1}^{J} L^{j} \nu_{i_j}\right)\right) , \quad v \in V\left(\Delta_M\right).
\end{align}
\end{itemize} 

\medskip
\noindent
The fractal $\mathcal{K}^{\left\langle \infty \right\rangle}$ is said to have the GLP if for some $M \in \mathbb{Z}$ there exist a labelling function $\ell_M$ satisfying both conditions above. Note that due to the self-similarity of this set, having this property for some $M$ gives the same for every $M \in \mathbb{Z}$. The GLP takes a quite simple form in the case of Sierpi\'nski triangle (cf. \cite{bib:KPP-PTRF, bib:KaPP}). However, in general, it is a rather delicate property (see \cite[Rem. 3.1]{bib:KOPP}). 

Now, for the unbounded fractal $\mathcal{K}^{\left\langle \infty \right\rangle}$ having GLP, we define a projection map 
$$
\mathcal{K}^{\left\langle \infty \right\rangle} \ni x \longmapsto \pi_{M}(x) \in \mathcal{K}^{\left\langle M \right\rangle}
$$
by the formula
\begin{equation}
\pi_M(x) = R_{\Delta_M}\left(x -\sum_{j=M+1}^{J} L^{j} \nu_{i_j}\right),
\end{equation}
where $\Delta_M = \mathcal{K}^{\left\langle M \right\rangle} + \sum_{j=M+1}^{J} L^{j} \nu_{i_j}$ is an $M$-complex containing $x$ and $R_{\Delta_M}$ is the unique rotation determined by \eqref{eq:rotation}. Here, the two cases are possible: \\
(1) if $x \notin V_M^{\langle \infty\rangle}$, then $\Delta_M = \Delta_M(x)$ (i.e. $\Delta_M$ can be chosen uniquely); \\
(2) if $x \in V_M^{\langle \infty\rangle}$, then $\Delta_M$ is a one of the $M$-complexes $\Delta^{(i)}_M$ such that 
$\left\{x\right\}= \bigcap_{i=1}^{r_x} \Delta^{(i)}_M$, where $r_x= \textrm{rank}(x)$ is the number of $M$-complexes meeting at $x$.

If $x$ is a vertex from $V_M^{\langle \infty\rangle},$  possibly belonging to more than one $M$-complex, then indeed we can choose any of those complexes in the definition above -- thanks to the GLP of $\mathcal{K}^{\left\langle \infty \right\rangle}$, the image does not depend on the particular choice of $\Delta^{(i)}_M$.

The projection $\pi_M$ is an essential tool to construct the reflected Brownian motion on $\mathcal{K}^{\left\langle M \right\rangle}$.

\subsection{Reflected Brownian motion on simple nested fractals}

We denote by $Z=(Z_t, \mathbf{P}^{x})_{t \geq 0, \, x \in \mathcal{K}^{\left\langle \infty \right\rangle}}$ the \emph{Brownian motion} on the USNF $\mathcal{K}^{\left\langle \infty \right\rangle}$. In the case of Sierpi\'nski gasket such a process has been rigorously constructed in \cite{bib:BP}. For general nested fractals, the Brownian motion has been first constructed on the unit fractal $\mathcal{K}^{\left\langle 0 \right\rangle}$ (\cite{bib:Lin}, see also \cite{bib:Kus}) and then extended to $\mathcal{K}^{\left\langle \infty \right\rangle}$ by means of Dirichlet forms (\cite{bib:F}, see also \cite{bib:KumKus, bib:Sh}). It is a strong Markov process with continuous paths, which distributions are invariant under local isometries of $\mathcal{K}^{\left\langle \infty \right\rangle}$. It has transition probability densities $g(t,x,y)$ with respect to the $d_f$-dimensional Hausdorff measure $\mu$ on $\mathcal{K}^{\left\langle \infty \right\rangle}$, i.e.,
$$
\mathbf{P}^{x}(Z_t \in A) = \int_A g(t,x,y) \mu(dy), \quad t > 0, \ \ x \in \mathcal{K}^{\left\langle \infty \right\rangle}, \ \ A \subset \cB(\mathcal{K}^{\left\langle \infty \right\rangle}),
$$
which are jointly continuous on $(0,\infty) \times \mathcal{K}^{\left\langle \infty \right\rangle} \times \mathcal{K}^{\left\langle \infty \right\rangle}$ and have the scaling property
$$
g(t,x,y) = L^{d_f} g(L^{d_w} t, L x, L y), \quad t>0, \quad x, y \in \mathcal{K}^{\left\langle \infty \right\rangle}.
$$
Moreover, there are absolute constants $c_{13},...,c_{16} >0$ such that the following subgaussian estimates \cite[Theorems 5.2, 5.5]{bib:Kum}
\begin{multline}
\label{eq:kum}
c_{13} t^{-d_s/2} \exp \left(-c_{14} \left(\frac{\left|x-y \right|^{d_w}}{t} \right)^{1 /\left(d_J -1\right)} \right) \leq g(t,x,y) \\
\leq c_{15} t^{-d_s/2} \exp \left(-c_{16} \left(\frac{\left|x-y \right|^{d_w}}{t} \right)^{1 /\left(d_J -1\right)} \right), \quad t>0, \ \ x,y \in \mathcal{K}^{\left\langle \infty \right\rangle}
\end{multline}
holds. The constant $d_J > 1$, called the \emph{chemical exponent} of $\mathcal{K}^{\left\langle \infty \right\rangle}$, is a parameter describing the shortest path scaling of the set $\mathcal{K}^{\left\langle \infty \right\rangle}$. Typically, $d_J \neq d_w$, but it is known that for the Sierpi\'nski gasket one has $d_J=d_w$. The above estimates were proven under the assumption that there exists $n \in \mathbb{N}$ such that for any $M\in\mathbb Z,$  if $x,y \in \mathcal{K}^{\left\langle \infty \right\rangle}$ satisfy $\left|x-y\right| \leq L ^{M}$, then $d_M(x,y) \leq n$ (\cite[Sec. 5]{bib:Kum}). It was shown in \cite{bib:KOPP} that in fact this assumption holds true for any planar simple nested fractal. For a fair account of the theory of Brownian motion on simple nested fractals we refer to \cite{bib:Bar} and
references therein.

The \emph{reflected Brownian motion} on $\mathcal{K}^{\left\langle M\right\rangle}$ was constructed in \cite{bib:KOPP} as a canonical projection of the free Brownian motion 
$$
Z_t^M = \pi_M(Z_t).
$$
Formally, this is the process $Z^M=(Z_t^M, \mathbf{P}^{x}_{M})_{t \geq 0,\, x \in \mathcal{K}^{\left\langle M\right\rangle}}$, where the measures $\mathbf{P}^{x}_{M}$ are determined by
\begin{align} \label{eq:fdd_reflected}
\mathbf{P}^{x}_{M}(Z^M_{t_1} \in A_1, ... , Z^M_{t_n} \in A_n)= \mathbf{P}^{x}(Z_{t_1} \in \pi_M^{-1}(A_1), ... , Z_{t_n} \in \pi_M^{-1}(A_n)),
\end{align}
for every $0 \leq t_1 < t_2 <...< t_n$, $x \in \mathcal{K}^{\left\langle M\right\rangle}$ and $A_1,...,A_n \in \cB(\mathcal{K}^{\left\langle M\right\rangle})$. As mentioned above, its transition probabilities are absolutely continuous with respect to the measure $\mu$ (restricted to $\mathcal{K}^{\left\langle M\right\rangle}$) with densities $g_M(t,x,y)$ given by
\begin{equation}
\label{eq:refldens}
g_{M}\left(t,x,y\right)= \left\{ \begin{array}{ll}
\displaystyle\sum_{y'\in \pi_{M}^{-1} (y)}{ g(t,x,y')} & \textrm{if } y \in \mathcal{K}^{\left\langle M\right\rangle} \backslash V_{M}^{\left\langle M\right\rangle}, \\
\displaystyle\sum_{y'\in \pi_{M}^{-1} (y)}{g(t,x,y')} \cdot \textrm{rank}(y') & \textrm{if } y \in V_{M}^{\left\langle M\right\rangle}, \\
\end{array}\right.
\end{equation}
where $\textrm{rank}(y')$ is the number of $M$-complexes meeting at the point $y'$. Moreover, it was proven in \cite[Th. 4.1 and Th.4.2]{bib:KOPP} that the function $g_M(t,x,y)$ is continuous in $(t,x,y)$ and symmetric and bounded in $(x,y)$ for every fixed $t>0$. This provides us with further regularity properties of the process $\left(Z_t^M\right)_{t \geq 0}$ such as Feller and strong Feller property. 

Our aim in the present paper is to find the sharp two-sided estimates for the densities $g_M(t,x,y)$ and for $g_M(t,x,y) - g(t,x,y)$. This goal will be achieved in the next section. 

\bigskip

\section{Estimates}

We are now in a position to state our main results in this paper. For given $c>0$ and for every $t>0$, $M \in \mathbb{Z}$ and $r\geq 0$ we denote
\begin{align}
f_{c}(t,r) = t^{-d_s/2} \exp \left(-c \left(\frac{r^{d_w}}{t} \right)^{\frac{1}{d_J-1}} \right)
\end{align}
and
\begin{align} \label{eq:def_h}
h_{c}(t,M) =L^{-d_f M} \left( \frac{L^{M}}{t^{1/d_w}} \vee 1\right)^{d_f-{\frac{d_w}{d_J-1}}} \exp\left(-c \left(  \frac{L^{M}}{t^{1/d_w}} \vee 1\right)^{\frac{d_w}{d_J-1}}\right).
\end{align}

\noindent
The first theorem gives the sharp two-sided estimates for $g_M(t,x,y)$.

\begin{theorem} \label{th:main}
Let $\mathcal{K}^{\left\langle \infty \right\rangle}$ be the USNF with the GLP. 
Then there exist constants $c_1, ... ,c_6 >0$ such that for every $t>0$, $M \in \mathbb{Z}$ and $x,y \in \mathcal{K}^{\left\langle M \right\rangle}$ one has
\begin{align*}
c_1 \left(f_{c_2}(t,|x-y|) \vee h_{c_3}(t,M) \right) & \leq g_M(t,x,y) \\ & \leq c_4 \left(f_{c_5}(t,|x-y|) \vee h_{c_6}(t,M)\right).
\end{align*}
\end{theorem}

\medskip
\noindent
We also obtain sharp two-sided bounds for the difference $g_M(t,x,y)-g(t,x,y)$. This result has direct important applications in our ongoing project.

\begin{theorem} \label{th:main2}
Let $\mathcal{K}^{\left\langle \infty \right\rangle}$ be the USNF with the GLP. 
Then there exist constants $c_{7}, ..., c_{12}>0$ such that for every $t>0$, $M \in \mathbb{Z}$ and $x,y \in \mathcal{K}^{\left\langle M \right\rangle}$ one has
\begin{align*}
c_{7}  \big(f_{c_{8}}(t,\delta_M(x,y)) \vee h_{c_{9}}(t,M)\big)
& \leq g_M(t,x,y)-g(t,x,y) \\
& \leq c_{10} \big(f_{c_{11}}(t,\delta_M(x,y)) \vee h_{c_{12}}(t,M)\big),
\end{align*}
where $\delta_M(x,y) = \inf_{z \in V^{*}_{M}} (|x-z|+|z-y|)$ with
$$
V^{*}_M := \left\{ z \in V_{M}^{\left\langle M\right\rangle} : \text{there exists} \ \Delta_M \in \mathcal{T}_M \ \text{such that} \ \Delta_M \cap \mathcal{K}^{\left\langle M\right\rangle}  = \{z\}\right\}.
$$
\end{theorem}

\bigskip
\noindent
We give the proofs of the above theorems after sequence of auxiliary results. First we fix some useful notation. 
For $M \in \mathbb{Z}$ and $y \in \mathcal{K}^{\left\langle M\right\rangle} \backslash V_{M}^{\left\langle M\right\rangle}$ we let
\begin{equation*}
A(M,m,y) = \left\{ y' \in \pi_M^{-1}(y) : y' \in \mathcal{K}^{\left\langle M+m+1\right\rangle} \backslash \mathcal{K}^{\left\langle M+m\right\rangle} \right\}, \quad m \geq 0,
\end{equation*}
and
\begin{equation*}
B(M,0,y) = \left\{ y' \in \pi_M^{-1}(y) : y' \in \mathcal{K}^{\left\langle M+1\right\rangle} \backslash \mathcal{K}^{\left\langle M\right\rangle}, \Delta_M\left(y'\right) \cap \mathcal{K}^{\left\langle M\right\rangle} = \emptyset \right\},
\end{equation*}
\begin{equation*}
C(M,0,y) = \left\{ y' \in \pi_M^{-1}(y) : y' \in \mathcal{K}^{\left\langle M+1\right\rangle} \backslash \mathcal{K}^{\left\langle M\right\rangle}, \Delta_M\left(y'\right) \cap \mathcal{K}^{\left\langle M\right\rangle} \neq \emptyset \right\},
\end{equation*}
so that
$$
A(M,0,y) = B(M,0,y) \cup C(M,0,y).
$$

Then we can decompose the fiber of $y$ as follows
\begin{equation*}
\pi_M^{-1}(y) = \bigcup_{m\geq 1} A(M,m,y) \cup B(M,0,y) \cup C(M,0,y) \cup \{y\},
\end{equation*}
and, consequently, for every $x,y \in \mathcal{K}^{\left\langle M\right\rangle} \backslash V_{M}^{\left\langle M\right\rangle}$,
\begin{align} \label{eq:decomposition}
g_M (t,x,y) & = \underbrace{\sum_{m\geq 1} \sum_{y' \in A(M,m,y)} g(t,x,y') + \sum_{y' \in B(M,0,y)} g(t,x,y')}_{=:g^{(1)}_M(t,x,y)} \nonumber \\ 
& \ \ \ \ \ \ \ \ \ \ \ \ \ \ \ \ \ \ \ \ \ \ \ \ \ \ \ \ \ \ \ \ \ \ \ \ \ \ \ \ \ \ \ \
+ \underbrace{\sum_{y' \in C(M,0,y)} g(t,x,y')}_{=:g^{(2)}_M(t,x,y)} + \, g(t,x,y). 
\end{align}
Note that $B(M,0,y)$ can be an empty set (the only planar example of $\mathcal{K}^{\left\langle \infty \right\rangle}$ with this property is the Sierpi\'nski gasket). Here we use the convention that the summation over an empty set always gives $0$. 

The following lemma will be used in proving our estimates for the function $g^{(1)}_M(t,x,y)$. 
It can be interpreted as the intrinsic growth property of the graph metric.   

\begin{lemma}[{\cite[Lem. A.2]{bib:KOPP}}]
\label{lem:metrics}
For every $M \in \mathbb{Z}$ and every $x,y \in \mathcal{K}^{\left\langle \infty \right\rangle}$ we have 
\begin{equation}
\label{eq:main_ass}
c_{17} L^{-M} \left|x-y\right| \leq d_M\left(x,y\right) \leq \max \left\{2, c_{18} N^{-M} \left|x-y\right|^{d_f} \right\},
\end{equation}
where $c_{17}$, $c_{18}$ are independent of $x$, $y$ and $M$.\\
\end{lemma}

The next two lemmas will be applied to get the upper bounds for the function $g^{(2)}_M(t,x,y)$ in the decomposition \eqref{eq:decomposition}. 

\begin{lemma}
\label{lem:const}
There exist a constant $c_{19} >0$ with the following property. For every $x,y \in \mathcal{K}^{\left\langle \infty \right\rangle}$ and $m \in \mathbb{Z}$ \, such that $x \in \Delta_m^{(1)}$, $y \in \Delta_m^{(2)}$ and $\Delta_m^{(1)} \cap \Delta_m^{(2)} = \emptyset$ we have $|x-y| \geq c_{19}L^m$.
\end{lemma}
\begin{proof} 
The lemma follows from \cite[Cor. A.1]{bib:KOPP} by scaling.
\end{proof}

\begin{lemma}
\label{lem:adjacent}
There exists an absolute constant $c_{20} >0$ with the following property.
If $x,y \in \mathcal{K}^{\left\langle M \right\rangle}$ and $y' \in \pi_M^{-1}(y) \backslash \{y\}$ is inside an $M$-complex $\Delta_M$ adjacent to $\mathcal{K}^{\left\langle M \right\rangle}$ such that $\Delta_M \cap \mathcal{K}^{\left\langle M \right\rangle} = \{z\}$ for some $z \in V_{M}^{\left\langle M\right\rangle}$, then
$$|x-y'| \geq c_{20} \left(|x-z| + |z-y|\right).$$
In particular, $|x-y'| \geq c_{20} |x-y|$. 
\end{lemma}

\begin{proof}
Assume first that $x, y \in \mathcal{K}^{\left\langle M \right\rangle} \backslash V_{M}^{\left\langle M\right\rangle}$ and let $y' \in \pi_M^{-1}(y)$ be such that $\Delta_M(y') \cap \mathcal{K}^{\left\langle M \right\rangle} = \{z\}$ with some $z \in V_{M}^{\left\langle M\right\rangle}$.

Observe that for every $m\leq M$ the vertex $z$ lies at the intersection of two $m$-complexes $\Delta_{m}^{(1)}(z)$ and $\Delta_{m}^{(2)}(z)$ such that $\Delta_{m}^{(1)}(z) \subseteq \mathcal{K}^{\left\langle M \right\rangle}$ and $\Delta_{m}^{(2)}(z) \nsubseteq \mathcal{K}^{\left\langle M \right\rangle}$.

Let now $m \in \mathbb{Z}$ be the smallest integer for which $x,y \in \Delta_{m}^{(1)}(z)$. Then we have $|x-z|, |y-z| \leq L^m c_1^{\prime}$, where $c_1^{\prime}$ is the diameter of any $0$-complex. On the other hand, as $x \notin \Delta_{m-1}^{(1)}(z)$ or $y \notin \Delta_{m-1}^{(1)}(z)$, we see that $x$ and $y'$ are in disjoint $(m-1)$-complexes. Those $(m-1)$-complexes are included in the two different $m$-complexes and at most one of them is attached to $z$. Then we get from Lemma \ref{lem:const} that $|x-y'| \geq L^{m-1} c_{19}$ and, in consequence,
\begin{equation*}
|x-z| + |z-y| \leq \frac{2 c_1^{\prime} L}{c_{19}} |x-y'|.
\end{equation*}
By the continuity of the Euclidean distance, the same is true for every $x,y \in \mathcal{K}^{\left\langle M \right\rangle}$. The second assertion follows from the triangle inequality for such a distance. The lemma holds with $c_{20} = c_{19}/(2c_1^{\prime} L)$.
\end{proof}

We now give the two-sided bounds for the function $g^{(1)}_M(t,x,y)$ which are the first crucial ingredient of the proofs of our main results.
This is the case when $y'$ under the sums in \eqref{eq:decomposition} are far away from $x$. 

\begin{lemma}
\label{lem:tail}
For every $t>0$, $M \in \mathbb{Z}$ and $x,y \in \mathcal{K}^{\left\langle M\right\rangle}  \backslash V_{M}^{\left\langle M\right\rangle}$, one has
\begin{align}
\label{eq:estimate1g}
c_{21} h_{c_{22}}(t,M) \leq g^{(1)}_M(t,x,y) \leq c_{23} h_{c_{24}}(t,M),
\end{align}
with certain numerical constants $c_{21}, ..., c_{24} >0$ (independent of $M$, $t$ and $x, y$).

\end{lemma}

\begin{proof}
Let $M \in \mathbb{Z}$, $t>0$ and $x,y \in \mathcal{K}^{\left\langle M\right\rangle} \backslash V_{M}^{\left\langle M\right\rangle}$.
We now prove the upper bound and the lower bound separately. \\

\noindent
THE UPPER BOUND. 
Observe that for every fixed $m \geq 1$ and $y' \in A(M,m,y)$ we have $d_{M+m-1}\left(x,y'\right) >2$. Then, by applying the upper bound in \eqref{eq:main_ass} for $d_{M+m-1}(x,y')$, we get 
$$\left|x-y'\right|^{d_f} \geq \frac{2}{c_{18}} N^{M+m-1}.$$
Moreover, the number of such points $y'$ is equal to the number of $M$-complexes inside $\mathcal{K}^{\left\langle M+m+1\right\rangle} \backslash \mathcal{K}^{\left\langle M+m\right\rangle}$, i.e. $N^{m} (N-1)$. Analogously, when $y' \in B(M,0,y)$, then $d_{M-1} \left(x,y'\right) >2$, which gives 
$$\left|x-y'\right|^{d_f} \geq \frac{2}{c_{18}} N^{M-1}.$$ 
There are less than $N-1$ of such points.

Then, by using the decomposition in \eqref{eq:decomposition}, the upper subgaussian estimate for $g(t,x,y')$ and the above observations, we get
\begin{align*}
& g^{(1)}_M(t,x,y) \\
& \ \leq c_1^{\prime} t^{-\frac{d_s}{2}} \sum_{m \geq 0} N^{m}(N-1) \exp\left( -c_2^{\prime} \left( \frac{\left(\frac{2}{c_{18}}\right)^{\frac{d_w}{d_f}} N^{(M+m-1)\frac{d_w}{d_f}}}{t} \right)^{\frac{1}{d_J -1}} \right) \\
& \ = c_1^{\prime} N^{-M+1}t^{-\frac{d_s}{2}} \sum_{m \geq 0} N^{M+m-1} (N-1)  \exp\left( -c_3^{\prime} \left( \frac{N^{(M+m)\frac{d_w}{d_f}} }{t} \right)^{\frac{1}{d_J -1}} \right),
\end{align*}
with an appropriate absolute positive constant $c_3^{\prime}$.
 
Since for any $\eta, \gamma>0$ we have
$$
\int_{N^{M-1}}^{\infty} e^{-\eta \xi^{\gamma}} d\xi = \sum_{m=0}^{\infty} \int_{N^{M-1+m}}^{N^{M+m}} e^{-\eta \xi^{\gamma}} d\xi \geq  \sum_{m=0}^{\infty} N^{M+m-1}(N-1) e^{- \eta N^{(M+m)\gamma}},
$$
the above series can be estimated above by an appropriate integral. We then get
\begin{align*}
g^{(1)}_M(t,x,y)
& \leq  \frac{c_4^{\prime}}{L^{M d_f} \ t^{d_s / 2}} \int_{N^{M-1}}^{\infty} \exp \left( -c_3^{\prime}  \left( \frac{\xi^{d_w / d_f}}{t} \right)^{\frac{1}{d_J -1}} \right) d\xi \\
& = \frac{c_4^{\prime}}{L^{M d_f} \ t^{d_f / d_w}} \int_{N^{M-1}}^{\infty} \exp \left(-c_3^{\prime}  \left( \frac{\xi^{1 / d_f}}{t^{1/d_w}} \right)^{\frac{d_w}{d_J -1}} \right) d\xi,
\end{align*}
which, by substitution $\xi^{1/d_f}t^{-1/d_w} = \zeta$, is equal to
$$
\frac{c_4^{\prime} d_f}{L^{M d_f}} \int_{N^{(M-1) /d_f} t^{-1/d_w}}^{\infty} \zeta^{d_f -1} \exp\left( -c_3^{\prime} \zeta^{\frac{d_w}{d_J-1}}\right) d\zeta.
$$
Now, by using an elementary estimate
\begin{equation*}
\int_{a}^{\infty} y^{\beta} e^{-\eta y^{\gamma}} dy \leq c (a \vee 1)^{\beta-\gamma+1} e^{-\eta(a \vee 1)^\gamma}, \quad \eta,\beta,\gamma >0,
\end{equation*}
and the fact that $N^{M/d_f} = L^M$, we can conclude the proof of the upper bound in \eqref{eq:estimate1g}, getting 
\begin{align*}
g^{(1)}_M(t,x,y)
& \leq c_5^{\prime} L^{-d_{f}M} \left( \frac{L^{M-1}}{t^{1/d_w}} \vee 1\right)^{d_f-\frac{d_w}{d_J-1}} \exp\left(-c_3^{\prime} \left( \frac{L^{M-1}}{t^{1/d_w}} \vee 1\right)^{\frac{d_w}{d_J-1}}\right) \\
& \leq c_6^{\prime}L^{-d_{f}M} \left( \frac{L^M}{t^{1/d_w}} \vee 1\right)^{d_f-\frac{d_w}{d_J-1}} \exp\left(-c_7^{\prime} \left( \frac{L^M}{t^{1/d_w}} \vee 1\right)^{\frac{d_w}{d_J-1}}\right).
\end{align*}

\noindent
THE LOWER BOUND. 
First recall that $B(M,0,y)$ can be an empty set. Therefore, we first write
\begin{align} \label{eq:low_aux}
g^{(1)}_M(t,x,y) \geq \sum_{m\geq 1} \sum_{y' \in A(M,m,y)} g\left(t,x,y'\right). 
\end{align}
When $m \geq 1$ and $y' \in  A(M,m,y)$, then $d_{M+m+1}\left(x,y'\right) =1$. By applying the lower estimate in \eqref{eq:main_ass} with $n=M+m+1$, we get 
$$\left|x-y'\right|\leq \frac{1}{c_{17}} L^{M+m+1}$$
(recall also that the cardinality of $A(M,m,y)$ is equal to $N^{m} (N-1)$). Therefore, by the lower subgaussian bound of $g(t,x,y')$, the series on the right hand side of \eqref{eq:low_aux} is larger than or equal to
\begin{align*}
c_8^{\prime}& t^{-\frac{d_s}{2}} \sum_{m \geq 1} N^{m}(N-1) \exp\left( c_9^{\prime} \left( \frac{\left(\frac{1}{c_{17}}\right)^{d_w} L^{(M+m+1)d_w}}{t} \right)^{\frac{1}{d_J-1}} \right) \\
& = c_8^{\prime} N^{-M}t^{-\frac{d_s}{2}} \sum_{m \geq 1} N^{M+m} (N-1)  \exp\left( c_{10}^{\prime} \left( \frac{N^{(M+m)\frac{d_w}{d_f}} }{t} \right)^{\frac{1}{d_J-1}} \right)
\end{align*}
where $c_8^{\prime}$ and $c_{10}^{\prime}$ are absolute constants. Now, by estimating the series by an appropriate integral (similarly as in the proof of the upper bound), we show that the above member is larger than or equal to
\begin{align*}
c_{11}^{\prime}L^{-M d_f}& t^{-d_s / 2} \int_{N^{M+1}}^{\infty} \exp \left( -c_{10}^{\prime}  \left( \frac{\xi^{d_w \ d_f}}{t} \right)^{\frac{1}{d_J -1}} \right) d\xi \\
& = c_{11}^{\prime}L^{-M d_f}t^{-d_s / 2} \int_{N^{M+1}}^{\infty} \exp \left(-c_{10}^{\prime} \left( \frac{\xi^{1/d_f}}{t^{1/d_w}} \right)^{\frac{d_w}{d_J -1}} \right) d\xi \\
& = c_{11}^{\prime} d_f L^{-M d_f} \int_{L^{(M+1)} t^{-1/d_w}}^{\infty} \zeta^{d_f-1} \exp\left( -c_{10}^{\prime} \zeta^{\frac{d_w}{d_J -1}}\right) d\zeta.
\end{align*}
Using an elementary estimate
\begin{equation*}
\int_{a}^{\infty} y^{\beta} e^{-\eta y^{\gamma}} dy \geq c (a \vee 1)^{\beta-\gamma+1} e^{-\eta(a \vee 1)^\gamma}, \quad \eta,\beta,\gamma >0,
\end{equation*}
we can now conclude the proof writing
\begin{align*}
& g^{(1)}_M(t,x,y) \\ 
& \ \ \ \ \ \ \geq c_{11}^{\prime} L^{-d_{f}M} \left(\frac{L^{(M+1)}}{t^{1/d_w}} \vee 1\right)^{d_f-\frac{d_w}{d_J-1}} \exp\left(-c_{10}^{\prime} \left(  \frac{L^{(M+1)}}{t^{1/d_w}} \vee 1\right)^{\frac{d_w}{d_J-1}}\right)\\
& \ \ \ \ \ \ \geq c_{12}^{\prime}L^{-d_f M} \left( \frac{L^{M}}{t^{1/d_w}} \vee 1\right)^{d_f-\frac{d_w}{d_J-1}} \exp\left(-c_{13}^{\prime} \left(  \frac{L^{M}}{t^{1/d_w}} \vee 1\right)^{\frac{d_w}{d_J-1}}\right).
\end{align*}
This also completes the proof of the lemma.
\end{proof}

We are now ready to collect all the above auxiliary estimates and to give the proofs of our main theorems.  

\begin{proof}(of Theorem \ref{th:main})
Let $M \in \mathbb{Z}$, $t>0$ and assume first that $x, y \in \mathcal{K}^{\left\langle M \right\rangle} \backslash V_{M}^{\left\langle M\right\rangle}$. Recall that from \eqref{eq:decomposition} we have 
\begin{equation*}
g_M (t,x,y) = g^{(1)}_M(t,x,y) + g^{(2)}_M(t,x,y) + g(t,x,y).  
\end{equation*}
By the subgaussian upper estimate in \eqref{eq:kum} and Lemma \ref{lem:adjacent}, we have
\begin{align}
0 \leq g^{(2)}_M(t,x,y) & \leq c_1^{\prime} \sum_{y' \in C(M,0,y)} t^{-d_s/2} \exp \left(-c_2^{\prime} \left(\frac{|x-y'|^{d_w}}{t} \right)^{\frac{1}{d_J -1}} \right) \nonumber \\ 
& \leq c_1^{\prime} k t^{-d_s/2} \exp \left(-c_2^{\prime} \left(\frac{(c_{20} |x-y|)^{d_w}}{t} \right)^{\frac{1}{d_J -1}} \right) \nonumber \\
& = c_3^{\prime} t^{-d_s/2} \exp \left(-c_4^{\prime} \left(\frac{\left|x-y \right|^{d_w}}{t} \right)^{\frac{1}{d_J -1}} \right).
\label{eq:neigh}
\end{align}
If $x, y \in \mathcal{K}^{\left\langle M \right\rangle} \backslash V_{M}^{\left\langle M\right\rangle}$, then the claimed two-sided bounds in 
Theorem \ref{th:main} follows from a combination of the estimates of $g^{(1)}_M$ in Lemma \ref{lem:tail}, the above estimates of $g^{(2)}_M$ and the subgaussian two-sided estimates of $g$ in \eqref{eq:kum}. Thanks to the continuity of the function $g_M(t,x,y)$ (see \cite[Th. 4.1 (1)]{bib:KOPP}) these bounds also extend to arbitrary $x, y \in \mathcal{K}^{\left\langle M \right\rangle}$. This completes the proof of the theorem.
\end{proof}

\begin{proof}(of Theorem \ref{th:main2})
Let $M \in \mathbb{Z}$, $t>0$ and suppose that $x, y \in \mathcal{K}^{\left\langle M \right\rangle} \backslash V_{M}^{\left\langle M\right\rangle}$. Similarly as above, we have 
\begin{align} \label{eq:proof2}
g_M(t,x,y)-g(t,x,y) = g^{(1)}_M(t,x,y)+g^{(2)}_M(t,x,y),
\end{align}
by \eqref{eq:decomposition}, and 
from Lemma \ref{lem:tail} we obtain that
\begin{equation*}
c_{21} h_{c_{22}}(t,M) \leq g^{(1)}_M(t,x,y) \leq c_{23} h_{c_{24}}(t,M).
\end{equation*}
It is then enough to estimate $g^{(2)}_M(t,x,y)$. From \eqref{eq:kum} and Lemma \ref{lem:adjacent}, we see that for $y' \in C(M,m,y)$
$$
g(t,x,y') \leq c_1^{\prime} f_{c_2^{\prime}} \left( |x-z|+ |z-y|\right),
$$
with $z \in V_{M}^{\left\langle M\right\rangle}$ such that $\Delta_M(y') \cap \mathcal{K}^{\left\langle M \right\rangle} = \{z\}$, and the absolute constants $c_1^{\prime}, c_2^{\prime}$. On the other hand, $|x-y'| \leq |x-z|+|z-y'| = |x-z|+|z-y|$, which gives
$$
g(t,x,y') \geq c_3^{\prime} f_{c_4^{\prime}} \left( |x-z| + |z-y|\right).
$$
Then, summing over $y' \in C(M,m,y)$, we obtain
\begin{align*}
g^{(2)}_M(t,x,y) = \sum_{y' \in C(M,m,y)} g(t,x,y') & \leq c_1^{\prime} \sum_{z \in V_M^{*}} f_{c_2^{\prime}} (t, |x-z|+|z-y|) \\
& \leq c_1^{\prime} k f_{c_2^{\prime}} (t, \delta_M(x,y))
\end{align*}
and
\begin{align*}
g^{(2)}_M(t,x,y) = \sum_{y' \in C(M,m,y)} g(t,x,y') & \geq c_3^{\prime} \sum_{z \in  V_M^{*}} f_{c_4^{\prime}}(t, |x-z|+|z-y|)  \\
& \geq c_3^{\prime} f_{c_4^{\prime}} (t, \delta_M(x,y)),
\end{align*}
where $\delta_M(x,y) = \inf_{z \in V^{*}_{M}} (|x-z|+|z-y|)$. As the functions on both sides of \eqref{eq:proof2} are continuous in $(x,y)$ on $\mathcal{K}^{\left\langle M \right\rangle} \times \mathcal{K}^{\left\langle M \right\rangle}$, the above bounds in fact extends to all $x, y \in \mathcal{K}^{\left\langle M \right\rangle}$. This completes the proof.
\end{proof}

Below we will write $a(t,x,y,M) \approx b(t,x,y,M)$ if there exist positive constants $c_1^{\prime}$, $c_2^{\prime}$ independent of $x,y,t,M$ such that 
\begin{equation*}
c_1^{\prime}a(t,x,y,M) \leq b(t,x,y,M) \leq c_2^{\prime} a(t,x,y,M).
\end{equation*}
We will now describe the behaviour of $g_M(t,x,y)$ in various time-space regimes. Recall that we have $|x-y| \leq L^M \diam(\mathcal{K}^{\left\langle 0 \right\rangle})$ for every $x, y \in \mathcal{K}^{\left\langle M \right\rangle}$ with $M \in \mathbb{Z}$. 

\begin{corollary} \label{cor:uniform}
For every $M \in \mathbb{Z}$ and $x, y \in \mathcal{K}^{\left\langle M \right\rangle}$ we have the following. 
If $t > L^{Md_w}$, then 
$$
g_M(t,x,y) \approx L^{-Md_f},  
$$
and if $0<t \leq L^{Md_w}$, then 
\begin{align} \label{eq:sec_case}
c_{25} f_{c_{26}}(t,|x-y|) \leq g_M(t,x,y) \leq c_{27} f_{c_{28}}(t,|x-y|),
\end{align}
with certain numerical constants $c_{25},...,c_{28}>0$ independent of $t, x, y$ and $M$. 
In particular, for $0<t \leq L^{Md_w}$ such that $t > |x-y|^{d_w}$ we have $g_M(t,x,y) \approx t^{-d_s / 2}$.
\end{corollary}

\begin{proof}
The first assertion follows directly from the fact that for $t > L^{Md_w}$ we have 
$$
t^{- d_s /2} < L^{-Md_{f}} \quad \text{and} \quad c_1^{\prime} \frac{|x-y|}{t^{1/d_w}} \leq \frac{L^M}{t^{1/d_w}} < 1.
$$
Indeed, the last two inequalities give 
$$
f_{c_2^{\prime}}(t,|x-y|) \approx t^{- d_s /2} \quad \text{and} \quad h_{c_3^{\prime}}(t,M) \approx L^{-Md_{f}}
$$
and from the estimates in Theorem \ref{th:main} we get $g_M(t,x,y) \approx L^{-Md_f}$. 

Consider now the case $|x-y|^{d_w} < t \leq L^{Md_w}$. Then again $|x-y|/ t^{1/d_w} \leq 1$, which yields
$$
f_{c_4^{\prime}}(t,|x-y|) \approx t^{- d_s /2} .
$$
Moreover,
\begin{align*}
0 \leq h_{c_5^{\prime}} (t,M) =  t^{-d_s /2} a^{-1} \exp \left(-c_{5}^{\prime} a\right),
\end{align*}
where $a=\left(\frac{L^M}{t^{1/d_w}} \right)^{d_w / (d_J -1)} \geq 1$. 
As the function $a \mapsto a^{-1} \exp \left(-c_5^{\prime} a\right)$ is bounded for $a \geq 1$, we conclude that 
$$
f_{c_{6}^{\prime}}(t,|x-y|) \vee h_{c_5^{\prime}}(t,M) \approx t^{-d_s /2}.
$$
This also implies \eqref{eq:sec_case}. 

Finally, when $0< t \leq L^{Md_w}$ and $t\leq |x-y|^{d_w}$, then by $|x-y|\leq c_7^{\prime} L^M$ and $\frac{L^M}{t^{1/d_w}}>1$, we see that
\begin{align*}
 0  \leq h_{c_8^{\prime}} (t,M) & = t^{-d_s /2} \left(\frac{L^M}{t^{1/d_w}} \right)^{-\frac{d_w}{d_J -1}} \exp \left(-c_8^{\prime} \left(\frac{L^M}{t^{1/d_w}} \right)^{\frac{d_w}{d_J -1}}\right)\\
& \leq t^{-d_s /2} \exp \left(-c_9^{\prime} \left(\frac{|x-y|}{t^{1/d_w}} \right)^{\frac{d_w}{d_J -1}}\right) = f_{c_9^{\prime}} (t, |x-y|).
\end{align*}
This again implies \eqref{eq:sec_case} and completes the proof. 
\end{proof}
Note that similar result can be given for the difference $g_M(t,x,y) - g(t,x,y)$.

\section*{Acknowledgements}
The author would like to thank Kamil Kaleta for drawing the attention to the subject and many valuable remarks.

\end{document}